\documentclass[11pt]{amsart}

\usepackage{amssymb}
\usepackage{overpic}
\usepackage{enumitem}   
\usepackage{graphicx} 
\usepackage{multicol}
\usepackage{xcolor}
\usepackage{amsrefs}
\usepackage[a4paper,margin=2cm]{geometry}
\usepackage[hidelinks]{hyperref}

\graphicspath{{Figures/}} 

\newtheorem{theorem}{Theorem}
\newtheorem{proposition}{Proposition} 
 
\newtheorem{remark}{Remark} 
\newtheorem{lemma}{Lemma}

\begin{document}
	
\title[Limit cycles of a quartic model]{On the limit cycles of\\ a quartic model for evolutionary stable strategies}

\author[Armengol Gasull, Luiz F. S. Gouveia and Paulo Santana]
{Armengol Gasull$^1$, Luiz F.S. Gouveia$^2$ and Paulo Santana$^3$}

\address{$^1$ Departament de Matem\`{a}tiques, Facultat de Ci\`{e}ncies, Universitat Aut\`{o}noma de Barcelona, 08193 Bellaterra, Barcelona, Spain}
\email{armengol.gasull@uab.cat}

\address{$^2$ UNICAMP, Campinas, Brazil \& UNESP, S. J. Rio Preto, Brazil}
\email{fernando.gouveia@unesp.br; lgouveia@unicamp.br}

\address{$^3$ UNESP, S. J. Rio Preto, Brazil}
\email{paulo.santana@unesp.br}



\begin{abstract} 
	This paper studies the number of centers and limit cycles  of the family of planar quartic polynomial vector fields  that has the invariant algebraic curve $(4x^2-1)(4y^2-1)=0.$ The main interest for this type of vector fields  comes from their appearance in some mathematical models in Game Theory composed by two players. In particular, we find examples with five nested limit cycles surrounding the same  singularity, as well as examples with four limit cycles formed by two disjoint nests, each one of them with two  limit cycles.  We also prove a Berlinski\u \i's type result for this family of vector fields.
\end{abstract}

\maketitle

\section{Main results}\label{Sec2}

This paper is devoted to study the number of centers and limit cycles  of the family of planar quartic polynomial vector fields  that has the invariant algebraic curve $(4x^2-1)(4y^2-1)=0.$ More qualitatively, we are interested on the periodic orbits of these systems included in the invariant square~$\Delta$ with boundaries formed by the four invariant lines, $x=\pm 1/2,$ $y=\pm 1/2.$

As we will see in detail in Section~\ref{Sec1} the main motivation for studying this family is its appearance in some mathematical models in Game Theory when the model has two players. Then, their strategies correspond to the coordinates of the points in this invariant square
\[\Delta=\left\{(x,y)\in\mathbb{R}^2\colon-1/2<x<1/2,\;-1/2<y<1/2\right\}.\]

More concretely, given $d\in\mathbb{N}$, let $\mathfrak{X}_d$ be the set of planar polynomial vector fields $X=(P,Q)$, with
\begin{equation}\label{7}
	P(x,y)=(4x^2-1)f(x,y), \quad Q(x,y)=(4y^2-1)g(x,y),
\end{equation}
where $f$, $g\colon\mathbb{R}^2\to\mathbb{R}$ are real polynomials of degree at most $d$. For these systems, $\Delta$ and the four lines described above are invariant. In this work we will center our attention to the cases $d\le2,$ and more specifically on the existence of periodic orbits in $\Delta.$ Of course, the cases $d>2$ deserve  a further study.

When $X\in\mathfrak{X}_1$ its phase portrait is well understood.  It has at most one singularity not lying in one of the four invariant lines.
Hence, $X$ can have at most one center and it is readily seen that it exists for some values of the parameters.  Moreover, it follows from Kooij \cite[Theorem $2.5.4$]{Kooij} that if $X\in\mathfrak{X}_1$, then $X$ can have at most one limit cycle and this upper bound is realizable. Furthermore, the phase portraits of these systems are well known, see for instance~\cite{BBS1,BBS2,SSHGM}. We will center our attention in the case $d=2,$ i. e. when $X$ is a quartic vector field. 

Recall that a singularity $p$ of a planar vector field  is a \emph{center} if there is a neighborhood $\mathcal{U}\subset\mathbb{R}^2$ of $p$ such that $\mathcal{U}\backslash\{p\}$ is filled by periodic orbits of the vector field. In our first main result, we study the maximum number of centers that $X\in\mathfrak{X}_2$ can have in $\Delta$ and we perturb these centers to obtain examples of systems with  many  limit cycles and with two  different configurations.

\begin{theorem}\label{T1}
	Let $X\in\mathfrak{X}_2$ as in~\eqref{7}. Then, the following statements hold.
	\begin{enumerate}[label=(\alph*)]
			\item $X$ has at most two centers in $\Delta.$ Moreover, there are $X$ with one or two centers in $\Delta.$ 
		\item There are $X$ with at least five nested limit cycles in $\Delta;$
		\item There are $X$ with at least four limit cycles in $\Delta,$ forming two disjoint nests of two limit cycles each.
	\end{enumerate}
\end{theorem}

The main idea for proving item $(a)$ is to use the same geometrical idea that the proof of the same result for quadratic systems: {\it the segment joining two singularities of center type is without contact by the flow,} see for instance~\cite{Cop1966}.  This is a direct consequence of Bezout's theorem in the quadratic case but it needs some extra work to be proved in our case. Once this geometric property is proved the result is straightforward because if three centers exist they must form a triangle and given any couple of these centers they must have opposite orientations giving rise to a contradiction.

The result in item $(b)$  gives an example of vector field in  $\mathfrak{X}_2$ with five limit cycles different to the one given in the recent paper~\cite{GraSza2024} with the same number of limit cycles. The way chosen in~\cite{GraSza2024} to obtain  five limit cycles has been the computations of the full Lyapunov quantities of a specific family of $X\in\mathfrak{X}_2.$ Our approach also uses Lyapunov quantities, but we rely on a different point of view, developed in~\cite{GinGouTor2021} and in its references, that starts with some centers and then perturb them to get Taylor's expansions of these constants in terms of these perturbative parameters.

We also study the type of singularities in $\Delta$ when $X\in \mathfrak{X}_2,$ in the richer dynamical  case where $X$ has four singularities in $\Delta.$ First we recall the classical and nice Berlinski\u \i's Theorem proved in 1960 for quadratic systems. To see its proof and more related references, see \cite{B,Cop1966,CGM1993}.

\begin{theorem}[Berlinski\u \i's Theorem (\cite{B})]
Suppose that $Y$ is a quadratic vector field with exactly four singularities. If the quadrilateral with vertices at these singularities is convex then two opposite singularities are saddles (index $-1$) and the other two are antisaddles (nodes, foci or centers, index $+1$). Otherwise their convex hull is a triangle and  then either the three exterior singularities at the vertices are saddles and the interior singularity an antisaddle, or vice-versa.
\end{theorem}

We prove next result for $X\in\mathfrak{X}_2.$  As we will see the tools used in its proof also give information when $X\in\mathfrak{X}_d$ for any $d.$ 

\begin{proposition}\label{T1.5} 
	Assume that $X\in\mathfrak{X}_2$ and it has four singularities in $\Delta.$ Then the conclusions of Berlinski\u \i's Theorem also hold for these four singularities.
\end{proposition}

It follows from Proposition~\ref{T1.5} and  Theorem~\ref{T1}\,$(a)$ that $\mathfrak{X}_2$ shares several qualitative properties with the family of the planar quadratic vector fields.  For instance both families have the same maximum number of centers and for both families Berlinski\u \i's Theorem holds. However, it is well-known \cites{Zeg2024,Zha2002} that quadratic vector fields cannot have two distinct nests of limit cycles with more than one limit cycle each. Therefore, it follows from Theorem~\ref{T1}\,$(c)$ that this particular property is not shared.	

As we will see in next section, one of the more interesting models in $\mathfrak{X}_d$, from the point of view of modelization, are the ones that when
we write
\begin{equation}\label{eq:fg}
	f(x,y)=\sum_{i+j=0}^{d}a_{ij}x^iy^j, \quad g(x,y)=\sum_{i+j=0}^{d}b_{ij}x^i y^j,\end{equation}
they satisfy that $a_{d,0}=b_{0,d}=0.$ For short, we will call the set of vector fields satisfying these two restrictions as $\mathfrak{X}_d^0\subset \mathfrak{X}_d.$
 In our second main result we will restrict our attention to $\mathfrak{X}_d^0$ for $d\le2.$

\begin{theorem}\label{T2}
	Let $X\in\mathfrak{X}_d^0$ as in~\eqref{7}, with $d\le 2,$ $f$ and $g$ given in \eqref{eq:fg}, and the restrictions $a_{d,0}=b_{0,d}=0.$ Then  the following statements hold.
	\begin{enumerate}[label=(\alph*)]
		\item  If $X\in\mathfrak{X}_1^0$, then $X$ has at most one center and it cannot have limit cycles;
		\item There are $X\in\mathfrak{X}_2^0$ with at least three nested limit cycles in $\Delta;$
		\item There are $X\in\mathfrak{X}_2^0$ with at least two limit cycles in $\Delta,$ each one of them surrounding a different singularity.
	\end{enumerate}
\end{theorem}

The study of upper bounds of the number of limit cycles  for vector fields either in  $\mathfrak{X}_2$ or in $\mathfrak{X}_2^0,$ even in particular cases, is an extremely difficult problem. As far as we know,  only some very particular subsystems in $\mathfrak{X}_2^0$ are studied in~\cite{MA} by transforming them into systems of Kukles type and then applying the Bendixson-Dulac criterion. In this work, we do not approach this problem.

The paper is organized as follows. Section~\ref{Sec1} is devoted to give a concise motivation of the interest for studying vector fields in $\mathfrak{X}_d$ or $\mathfrak{X}_d^0$ in Game Theory. In Section~\ref{Sec3} we present the main tools that will be used throughout this paper. In Section~\ref{Sec4} we study some center conditions in $\mathfrak{X}_2$ and $\mathfrak{X}_d^0$. The bifurcation of limit cycles from one singularity is studied in Section~\ref{Sec5}. In Section~\ref{Sec6} we study the simultaneously bifurcation of limit cycles from two singularities. Finally, the main results are proved in Section~\ref{Sec7}.
	
\section{Vector fields in $\mathfrak{X}_d$ and Game Theory}\label{Sec1}

In 1973 Smith and Price \cite{SmiPri1973} introduced concepts of Game Theory into the Biology, coining the notion of \emph{Evolutionary Stable Strategies} (ESS). Roughly speaking, given a game between two or more players (modeling for example a conflict between different species of animals), an ESS is a strategy such that if most of the players follows it, then there is no disruptive strategy that would give higher advantages for the other players. That is, the best strategy for the other players is to also follow the ESS.  In other words \cite{Frid}, an ESS is a ``uninvadable'' state of the population, meaning that if the populations are already in an ESS, then any deviant behavior will eventually disappear under natural selection.

In 1978 Taylor and Jonker \cite{TayJon1978} introduced concepts of Ordinary Differential Equations (ODE) into the field of ESS in such a way that the game is now modelled by a system of ODE, see also~\cite{HS}. In particular, if the game has two players then the models result in a planar polynomial vector field. Under these concepts Hofbauer et al~\cite{HSS} proved in the planar case (and Palm~\cite{Palm} in the higher dimensional case) that every ESS is an asymptotically stable singularity, with the converse not holding.

However as observed by Hofbauer et al~\cite{HSS}, under the language  of ODE there is nothing special distinguishing ESS from asymptotic stability. Hence, it would be more appropriate to study asymptotic stability rather than ESS. In particular, in this context a stable limit cycle can be interpreted as an \emph{oscillating stable strategy}. That is, an ``uninvadable'' oscillating state of the population which eradicates any deviant behavior.

Therefore as anticipated in the introduction, since the simplest version of the game (i.e. $\mathfrak{X}_1$) can have at most one limit cycle (and one singularity in $\Delta$), in this paper we allow more complexity to the game and provide a first approach for the number and the configuration of its oscillating states. 

It is also noteworthy to observe that since every asymptotically stable singularity has index $+1$, Proposition~\ref{T1.5} also provides an upper bound and the possible configurations of the \emph{stable strategies} of the game. 

We now present a briefly introduction of the model in the two players situation. Let $\Gamma_1$ and $\Gamma_2$ be two players and $\{X_1,X_2\}$, $\{Y_1,Y_2\}$ be their \emph{pure strategies}. We denote by $a_{ij}^*\in\mathbb{R}$ the \emph{payoff} of the pure strategy $X_i$ against the pure strategy $Y_j$, and by $b_{ij}^*\in\mathbb{R}$ the payoff of $Y_i$ against $X_j$. Given a probabilistic vector of dimension two
	\[x=(x_1,x_2)\in S^2:=\{(x_1,x_2)\in\mathbb{R}^2\colon x_1\geqslant0,\; x_2\geqslant0,\; x_1+x_2=1\},\]
we associate to it the \emph{mix strategy} $x_1X_1+x_2X_2$. Similarly, given $y\in S^2$, we associate the mix strategy $y_1Y_1+y_2Y_2$. Let
	\[A^*=\left(\begin{array}{cc} a_{11}^* & a_{12}^* \\ a_{21}^* & a_{22}^* \end{array}\right), \quad B^*=\left(\begin{array}{cc} b_{11}^* & b_{12}^* \\ b_{21}^* & b_{22}^* \end{array}\right),\]
be the \emph{payoff matrices}. Given $y\in S^2$ we define
\begin{equation}\label{2}
	\left<e_i,A^*y\right>=a_{i1}^*y_1+a_{i2}^*y_2,
\end{equation}
as the payoff of $X_i$ against $y_1Y_1+ y_2Y_2$ (here, $\left<\cdot,\cdot\right>$ denotes the standard inner product of $\mathbb{R}^2$). Similarly, given $x\in S^2$, the payoff of the pure strategy $Y_i$ against the mix strategy $x_1X_1+x_2X_2$ is given by,
\[\left<e_i,B^*x\right>=b_{i1}^*x_1+b_{i2}^*x_2.\]
Furthermore, the \emph{average payoff} of $x_1X_1+x_2X_2$ against $y_1Y_1+y_2Y_2$ is given by,
\begin{equation}\label{3}
	\left<x,A^*y\right>=a_{11}^*x_1y_1+a_{12}^*x_1y_2+a_{21}^*x_2y_1+a_{22}^*x_2y_2,
\end{equation}
while the average payoff of $y_1Y_1+y_2Y_2$ against $x_1X_1+x_2X_2$ is given by,
\[\left<y,B^*x\right>=b_{11}^*x_1y_1+b_{12}^*x_1y_2+b_{21}^*x_2y_1+b_{22}^*x_2y_2.\]
The dynamic between players $\Gamma_1$ and $\Gamma_2$ is defined by the ordinary system of differential equations,
\begin{equation}\label{4}
	\begin{array}{ll }
		\dot x_1=x_1\bigl(\left<e_1,A^*y\right>-\left<x,A^*y\right>\bigr), & \quad\dot y_1=y_1\bigl(\left<e_1,B^*x\right>-\left<y,B^*x\right>\bigr), \vspace{0.2cm} \\
		\dot x_2=x_2\bigl(\left<e_2,A^*y\right>-\left<x,A^*y\right>\bigr), & \quad\dot y_2=y_2\bigl(\left<e_2,B^*x\right>-\left<y,B^*x\right>\bigr). 
	\end{array}
\end{equation} 
In other words, the rate of change of the population $x_i$ depends on the difference between the payoffs \eqref{2} and \eqref{3} (the bigger the difference, the bigger the superiority of strategy $X_i$), and on the size of the population itself, as a percentage of the total population $x_1+x_2$. Since $x_1+x_2=y_1+y_2=1$, it follows that to describe the dynamic between players $\Gamma_1$ and $\Gamma_2$, it is necessary only two variables, namely $x=x_1$ and $y=y_1$. Therefore, if we replace $x_2=1-x$ and $y_2=1-y$, then we can simplify system \eqref{4} and thus obtain the planar ordinary system of differential equations given by
\begin{equation}\label{5}
	\begin{array}{l}
		\dot x=x(x-1)\bigl(a_{22}^*-a_{12}^*+(a_{12}^*+a_{21}^*-a_{11}^*-a_{22}^*)y\bigr), \vspace{0.2cm} \\
		\dot y=y(y-1)\bigl(b_{22}^*-b_{12}^*+(b_{12}^*+b_{21}^*-b_{11}^*-b_{22}^*)x\bigr). 
	\end{array}
\end{equation}
Applications of such model can be found in a myriad of research areas, such as Parental Investing, Biology, Genetics, Evolution, Ecology, Politics and Economics. See \cites{Hines,MH,SchSig,H,Bomze,AccMarOvi,CanBer,SSHW,XiaoYu} and the references therein. The chasing for more realistic models demand the payoff coefficients to depend on the weight given to strategies $X_i$ and $Y_j$, rather than being constant, i.e. $a_{ij}^*=a_{ij}^*(x,y)$ and $b_{ij}^*=b_{ij}^*(x,y)$. Hence, if we assume that each payoff $a_{ij}^*$, $b_{ij}^*$ is a polynomial of degree at most $d$, then system \eqref{5} becomes
\begin{equation}\label{6}
	\dot x=x(x-1)f^*(x,y), \quad \dot y=y(y-1)g^*(x,y),
\end{equation}
with $f^*$ and $g^*$ polynomials of degree at most $d+1$. Let
	\[Q=\{(x,y)\in\mathbb{R}^2\colon 0<x<1,\;0<y<1\},\]
and observe that the model is well defined only in the closed square $\overline{Q}$ (where $\overline{Q}$ denotes the topological closure of $Q$). Since we are interested only in the dynamics inside $\overline{Q}$, we can translate the origin of~\eqref{6} to the point $\left(-{1}/{2},-{1}/{2}\right)$ and thus, after a change in the time variable, obtain system~\eqref{7}, that has the invariant square $\Delta.$
	
Observe that if each payoff function $a_{ij}^*$, $b_{ij}^*$ is a polynomial of degree $d-1$, then $f$ and $g$ are polynomials of degree $d$. However, it follows from \eqref{5} that not every $X\in\mathfrak{X}_d$ is realizable by some choice of payoff functions of degree $d-1$. More precisely,  $f$ and $g$ of the form \eqref{eq:fg} are  realizable by some choice of payoff functions of degree $d-1$ if, and only if, $a_{d,0}=0$ and $b_{0,d}=0$. This is precisely the class  $X\in\mathfrak{X}_d^0$ introduced before Theorem~\ref{T2}.

\section{Preliminaries}\label{Sec3}

\subsection{Lyapunov quantities}\label{Sec3.1}

Let $X$ be a planar polynomial vector field. We say that a non-constant analytic function $H\colon\mathbb{R}^2\to\mathbb{R}$ is a \emph{first integral} of $X$ if the orbits of $X$ are contained in the level sets of $H$. That is, if $(x(t),y(t))$ is an orbit of $X$, $t\in I$, then there is $c\in\mathbb{R}$ such that $H(x(t),y(t))=c$, for all $t\in I$. Observe that $H$ is a first integral of $X$ if, and only if,
\begin{equation}\label{23}
	P\frac{\partial H}{\partial x}+Q\frac{\partial H}{\partial y}=0.
\end{equation}
Suppose now that $X$ has a singularity at the origin. Let also $DX(x,y)$ denote the Jacobian matrix of $X$ at the point $(x,y)\in\mathbb{R}^2$. We say that the origin is a \emph{linear center} if $\det DX(0,0)>0$ and $\operatorname{Trace}(DX(0,0))=0$. That is, the origin is a linear center if the eigenvalues of $DX(0,0)$ are given by $\pm i\beta$, where $i$ is the complex unity satisfying $i^2=-1$ and $\beta>0$. 
Hence, after a linear change of variables and time we can suppose that $X=(P,Q)$ is given by
\begin{equation}\label{21}
	P(x,y)=-y+F(x,y), \quad Q(x,y)=x+G(x,y),
\end{equation}
with $F$ and $G$ without neither constants nor linear terms. Suppose that the origin is a center of \eqref{21}. Notice that the linear part of \eqref{21} has nonzero determinant and   then it is said that the singularity is a \emph{non-degenerated} center of \eqref{21}. Except in Theorem~\ref{T1}, in this paper by \emph{center} we will mean a non-degenerated center.

\begin{theorem}[Poincar\'e-Lyapunov Theorem]\label{PL}
	Let $X=(P,Q)$ be given by \eqref{21}. Then the origin is a center if, and only if, it admits a local first integral of the form,
	\begin{equation}\label{22}
		H(x,y)=x^2+y^2+\sum_{i+j=3}^{\infty}c_{ij}x^iy^j.
	\end{equation}
	Moreover, the existence of a formal first integral of the above form implies in the existence of a local first integral of the same form.
\end{theorem}
 
For a proof of Theorem~\ref{PL} see  \cite[Theorem $12.6$]{IlyYak2008} or \cite[Theorem $3.2.9$]{RomSha2009}. The necessary conditions for having a center at the origin is obtained by looking for a formal power series of the form \eqref{22} satisfying \eqref{23}. More precisely, to obtain the first $N$ necessary conditions we write down \eqref{22} up to order $2N+2$,
\begin{equation}\label{24}
	H(x,y)=x^2+y^2+\sum_{i+j=3}^{2N+2}c_{ij}x^iy^j.
\end{equation}
Then we replace \eqref{24} at \eqref{23} and, for each $k\in\{3,\dots,2N+2\}$, we equate to zero the coefficients of the monomials of degree $k$. Then, starting at $k=3$, we solve the equations recurrently for each $k\in\{3,\dots,2N+2\}$. It can be seen, that it is always possible to obtain a formal power series of the form \eqref{22} satisfying,
\begin{equation}\label{25}
	P\frac{\partial H}{\partial x}+Q\frac{\partial H}{\partial y}=\sum_{k=1}^{\infty}V_k(x^2+y^2)^{k+1},
\end{equation}
see for instance Songling~\cite{Son1984}. The constant $V_k\in\mathbb{R}$ is the $k$-th \emph{Lyapunov constant}. If $V_k=0$ for every $k\geqslant1$, then it follows from \eqref{25} that $H$ is a formal first integral and thus it follows from Theorem~\ref{PL} that the origin is a center. On the other hand, if some $V_k\neq0$, then the polynomial obtained from \eqref{24} (i.e. by truncating $H$ for some $N\geqslant1$ big enough) is a \emph{Lyapunov function} (see \cite[Section $1.8$]{DumLliArt2006}) and thus the origin is a focus. In this case, the stability of the focus is determined by the sign of $V_{k_0}$, where $k_0\geqslant1$ is the first index such that $V_{k_0}\neq0$. Although the constants $c_{ij}$ of \eqref{24} with $i+j=k$ and $k$ even cannot be uniquely determined, and thus the Lyapunov constants themselves cannot be uniquely determined, it known that $V_k$ is uniquely determined modulo the ideal generated by $V_1,\dots, V_{k-1}$. More precisely, let $X$ be given by \eqref{21} and let,
	\[F(x,y)=\sum_{i+j=2}^{d}a_{ij}x^iy^j, \quad G(x,y)=\sum_{i+j=2}^{d}b_{ij}x^iy^j.\]
Let $\mathcal{I}=\mathbb{R}[a_{20},a_{11},\dots,b_{1,d-1},b_{0,d}]$ be the ring of real polynomials in the coefficients of $X$. Let also $V_1,\dots, V_{k-1}$ be a choice of Lyapunov constants and let $I=\left<V_1,\dots, V_{k-1}\right>\subset\mathcal{I}$ be the ideal generated by them. If $V_k$ and $V_k'$ are two choices of the $k$-th Lyapunov constant, then $V_k-V_k'\in I$. In particular, if $V_1=\dots=V_{k-1}=0$, then $V_k=0$ if, and only if, $V_k'=0$. See again Songling~\cite{Son1984}. 

The center conditions can be obtained by a different, but very related way. For every $k\in\mathbb{N}$, $k\geqslant2$, let $\psi_{2k}\colon\mathbb{R}^2\to\mathbb{R}$ be a polynomial of homogeneous degree $2k$ such that,
	\[\int_{0}^{2\pi}\psi_{2k}(\cos\theta,\sin\theta)\;d\theta\neq0.\]
It follows from Suba \cite[Theorem $4.1$]{Suba1997} that instead of \eqref{25} we can seek for a formal first integral $H$ satisfying,
\begin{equation}\label{26}
	P\frac{\partial H}{\partial x}+Q\frac{\partial H}{\partial y}=\sum_{k=2}^{\infty}L_{k-1}\psi_{2k}.
\end{equation}
The constant $L_k\in\mathbb{R}$ is the $k$-th \emph{Lyapunov quantity} (observe that it depends on the choice of the sequence $\{\psi_{2k}\}_{k\geqslant2}$). It follow from Suba \cite[Section $5$]{Suba1997} that $L_1=\dots=L_{k-1}=0$ and $L_k\neq0$ if, and only if, $V_1=\dots=V_{k-1}$ and $V_k\neq0$. Therefore, if $L_k=0$ for every $k\in\mathbb{N}$, then $V_k=0$ for every $k\in\mathbb{N}$ and thus the origin is a center. On the other hand, if there is a minimal $k\in\mathbb{N}$ such that $L_k\neq0$, then $V_k\neq0$ and thus the origin is a focus. In this latter case, although the origin is a focus, we observe that the polynomial obtained by truncating the formal solution $H$ of \eqref{26} is not necessarily a Lyapunov function. 

It follows from \eqref{26} that the quantities $L_k$ are polynomials in the coefficients of $X$, i.e. $L_k\in\mathcal{I}$, for every $k\in\mathbb{N}$. Hence, for $N$ big enough and for a given set of functions $\{\psi_{2k}\}_k,$ the above algorithm determines also a set of necessary algebraic conditions $\{L_k=0\colon k\in\{1,\dots,N\}\}$ for the origin be a center. With this collection of necessary conditions in hands, we then apply a set of tools (which will be described in Sections~\ref{Sec3.3}, \ref{Sec3.4} and~\ref{Sec3.4b}) to discovery whether these conditions are also sufficient for the origin to be a center or not. Similar to Cruz et al \cite{CruRomTor2020}, to simplify the calculations, in this paper we consider $\psi_{2k}(x,y)=x^{2k}$ and thus we seek for a formal series $H$ satisfying,
\begin{equation*}
	P\frac{\partial H}{\partial x}+Q\frac{\partial H}{\partial y}=\sum_{k=1}^{\infty}L_kx^{2k+2}.
\end{equation*}
It follows from Suba \cite[p. $9$]{Suba1997} that for this particular choice of $\psi_{2k}$, if there is a minimal $k\in\mathbb{N}$ such that $L_k\neq0$ (and thus $V_k\neq0$), then $\operatorname{sgn}(L_k)=\operatorname{sgn}(V_k)$. Therefore, the  origin is a stable (resp. unstable) focus if $L_k<0$ (resp. $L_k>0$). 

\subsection{Bifurcation of limit cycles}\label{Sec3.2}

Let $X_{\mu}=(P_c(x,y;\mu),Q_c(x,y;\mu))$ be a family of vector fields of degree $n$ depending on a parameter $\mu\in\mathbb{R}^m$, such that the origin is always a center. Consider the polynomial perturbation $Y_{\mu,\alpha,\Lambda}=(P,Q)$ of $X_{\mu}$ given by
\begin{equation}\label{27}
	\begin{array}{l}
		P(x,y;\mu,\alpha,\Lambda)=P_c(x,y;\mu)+\alpha y+F(x,y;\alpha,\Lambda), \vspace{0.2cm} \\ Q(x,y;\mu,\alpha,\Lambda)=Q_c(x,y;\mu)+\alpha x+G(x,y;\alpha,\Lambda),
	\end{array}
\end{equation}
with $\alpha\in\mathbb{R}$, $\Lambda\in\mathbb{R}^p$ and such that $F$ and $G$ are polynomials of degree $n$ in $x$ and $y$, with neither constant nor linear terms, and such that $F(x,y;0,0)=G(x,y;0,0)=0$. More precisely, set 
	\begin{equation}\label{eq:per} F(x,y;0,\Lambda)=\sum_{i+j=2}^{n}a_{ij}x^iy^j, \quad G(x,y;0,\Lambda)=\sum_{i+j=2}^{n}b_{ij}x^iy^j,\end{equation}
with $\Lambda=(a_{20},a_{11},\dots,b_{1,n-1},b_{0,n})\in\mathbb{R}^p$ with $p={n^2+3n-4}.$

\begin{theorem}[Theorem $3.1$ of \cite{GinGouTor2021}]\label{GGT}
	For $\alpha=0$, let $L_j=L_j(\mu,\lambda)$ be $j$-th Lyapunov quantity of system~\eqref{27} when $\Lambda=M \lambda,$ where $\lambda\in\mathbb{R}^k$ and $M=M(\mu)$ is a  $p\times k$ matrix depending on $\mu$ and with rank~$k$. Assume $k$, $\ell$ and $M$ are such that
	\begin{equation}\label{28}
		L_j(\mu,\lambda)=\left\{\begin{array}{l}
								\displaystyle \lambda_j + O_2(\lambda), \text{ for } j\in\{1,\dots,k-1\}, \vspace{0.2cm} \\
								\displaystyle \sum_{l=1}^{k-1}g_{j,l}(\mu)\lambda_l+f_{j-k}(\mu)\lambda_k+O_2(\lambda), \text{ for } j\in \{k,\dots,k+\ell\},
							\end{array}\right.
	\end{equation}
	where $O_2(\lambda)$ denote all the monomials of degree equal to or higher than $2$ in $\lambda\in\mathbb{R}^k$, with coefficients being function on the parameter $\mu\in\mathbb{R}^m$. If there exists $\mu_0\in\mathbb{R}^m$ such that $f_0(\mu_0)=\dots=f_{\ell-1}(\mu_0)=0$, $f_\ell(\mu_0)\neq0$ and the jacobian matrix of $(f_0,\dots,f_{\ell-1})$ has rank $\ell$ at $\mu=\mu_0$, then system \eqref{27} has $k+\ell$ small amplitude limit cycles bifurcating from the origin.
\end{theorem}

\begin{remark}
	To apply Theorem~\ref{GGT} the most difficult part is to know the values $k$, $\ell$ and the linear change of parameters given by the matrix $M.$ In fact, they can not be determined a priory and they are obtained from the computations of the Taylor expansions of order one  of the Lyapunov quantities in the parameters $\Lambda\in\mathbb{R}^p,$ when $\alpha=0.$ In fact, it can also be applied when  $F$ or $G$ are not all polynomials of degree $n$ given in~\eqref{eq:per} but $\Lambda$ is any linear vectorial subspace of dimension $q,$ with  $k\le q<p.$ Finally, observe that $k$ also can be $1$ and in this situation the first row in~\eqref{28} is empty and all the functions $g_{j,l}$ are identically 0.
\end{remark}

\begin{remark}
	Our statement of Theorem~\ref{GGT} is adapted to our interests and  it is a little bit more general that the one given in~\cite{GinGouTor2021}.  We have added the dependence on $\alpha$ in the functions $F$ and $G.$ Nevertheless, its proof is essentially the same because it  relies on what happens when $\alpha=0.$ We have introduced this variation because, as we will see in Sections~\ref{Sec5} and~\ref{Sec6}, we need that our perturbations of a given system with $\Delta$ invariant have this set invariant as well. 
\end{remark}

\subsection{Reversibility}\label{Sec3.3}

Let $X$ be a $C^k$-planar vector field and $\varphi\colon\mathbb{R}^2\to\mathbb{R}^2$ a $C^k$-diffeomorphism such that $\varphi=\varphi^{-1}$ (also known as \emph{involution}, i.\,e. $\varphi^{2}=\operatorname{Id}$). We say that $X$ is $\varphi$-reversible if,
\begin{equation}\label{18}
	D\varphi(x)X(x)=-X(\varphi(x)),
\end{equation}
for every $x\in\mathbb{R}^n$. In simple words, $X$ is $\varphi$-reversible if after we apply the change of variables $x\mapsto\varphi(x)$ at $X$, we obtain $-X$. We denote by $S=\{x\in\mathbb{R}^n\colon\varphi(x)=x\}$  the set of the fixed points of $\varphi$. The simpler involutions are
	\[\varphi_1(x,y)=(-x,y), \quad \varphi_2(x,y)=(x,-y), \quad \varphi_3(x,y)=(y,x), \quad \varphi_4(x,y)=(-y,-x),\]
and their corresponding sets of the fixed points are
\[\begin{array}{ll}
	S_1=\{(x,y)\in\mathbb{R}^2\colon x=0\}, & \quad S_2=\{(x,y)\in\mathbb{R}^2\colon y=0\}, \vspace{0.2cm} \\
	S_3=\{(x,y)\in\mathbb{R}^2\colon y=x\}, & \quad S_4=\{(x,y)\in\mathbb{R}^2\colon y=-x\}.
\end{array}\]

We will use the following classical result.

\begin{theorem}[Poincaré reversibility criterion] 
	If $X$ is a $C^1$-planar  $\varphi_k$-reversible vector field, where $\varphi_k$ is one the above four involutions, and it has a periodic orbit intersecting $S_k,$ then it cannot be a limit cycle. In particular, if $p\in S_k$ and it is a singularity of type linear center  then $p$ is a center.
\end{theorem}	

It  can be extended by using other involutions (see \cite[Section 2]{Tei1997Cod1}), or for  more general   planar vector fields after introducing a wider class of reversible vector fields (\cite{BBT}), or in  higher dimensions (\cite{MedTei1,MedTei2}).

Suppose that $X$ has a linear center at the origin. It follows from Poincaré reversibility criterion that if $X$ is $\varphi_i$-reversible, for some $i\in\{1,2,3,4\}$, then the origin is a center. Since $S_1$ is given by the straight line $x=0$, if $X$ is $\varphi_1$-reversible, then we say that $X$ is \emph{reversible in relation to} $x=0$. Similarly, in the other cases. It follows from \eqref{18} that:
\begin{enumerate}[label=(\alph*)]
	\item $X=(P,Q)$ is reversible in relation to $x=0$ if, and only if,
		\[P(x,y)=P(-x,y), \quad Q(x,y)=-Q(-x,y);\]
	\item $X=(P,Q)$ is reversible in relation to $y=0$ if, and only if,
		\[P(x,y)=-P(x,-y), \quad Q(x,y)=Q(x,-y);\]
	\item $X=(P,Q)$ is reversible in relation to $y=x$ if, and only if,
		\[P(x,y)=-Q(y,x), \quad Q(x,y)=-P(y,x);\]
	\item $X=(P,Q)$ is reversible in relation to $y=-x$ if, and only if,
		\[P(x,y)=Q(-y,-x), \quad Q(x,y)=P(-y,-x).\]
\end{enumerate}
Therefore, if $X$ has a linear center at the origin, then a useful first approach is to check if $X$ satisfies one of the statements $(a)$, $(b)$, $(c)$ or $(d)$. If so, then the origin is a center.

\subsection{Darboux Theory of Integrability}\label{Sec3.4}

Let $X=(P,Q)$ be a planar polynomial vector field and let $R\colon\mathbb{R}^2\to\mathbb{R}$ be a non-constant analytic function. We say that $R$ is an \emph{integrating factor} of $X$ if,
\[\frac{\partial (PR)}{\partial x}+\frac{\partial (QR)}{\partial y}=0.\]
If $X$ has an integrating factor $R$, then $X$ has a first integral (see \cite[Section $8.2$]{DumLliArt2006}). Therefore, if $X$ has a linear center at the origin and if it admits an integrating factor in some neighborhood $U\subset\mathbb{R}^2$ of the origin, then it is a center. 

As for many problems it is not an easy task to find explicitly integrating factors.
The Darboux Theory of Integrability provides a systematic way to search for them when $X$ is polynomial and it  has several invariant algebraic curves satisfying some specific properties. Recall that it is said that $F(x,y)=0$ is an invariant algebraic invariant curve for $X$ if there exists a polynomial $K\colon\mathbb{R}^2\to\mathbb{R}$ such that,
\begin{equation*}
	P\frac{\partial F}{\partial x}+Q\frac{\partial F}{\partial y}=KF.
\end{equation*}
Then, $K$ is called the \emph{cofactor} of $F$. Observe that if $X$ has degree $d$ (i.e. $d=\max\{\deg P,\deg Q\}$), then $\deg K=d-1$. Notice that whenever the set $\{F=0\}\subset\mathbb{R}^2$ is not empty then it is invariant by the flow of $X.$ 

The result of this theory that we will use in this work is the next theorem. To see its proof or some improvements and other related results see for instance~\cite[Chapter $8$]{DumLliArt2006}.  

\begin{theorem}[Theorem $8.7$ of \cite{DumLliArt2006}]\label{Darboux}
	Let $X$ be a planar polynomial vector field and let $F_1,\dots,F_n$ be invariant algebraic curves of $X$ with cofactors $K_1,\dots, K_n$. If there are $\lambda_1,\dots,\lambda_n\in\mathbb{R}$ such that
	\begin{equation*}
		\frac{\partial P}{\partial x}+\frac{\partial Q}{\partial y}+\sum_{i=1}^{n}\lambda_i K_i=0,
	\end{equation*}
	then $R=F_1^{\lambda_1}\cdot\dots\cdot F_n^{\lambda_n}$ is an integrating factor of $X$.
\end{theorem}

\subsection{A simple family without limit cycles and centers}\label{Sec3.4b}

For completeness we prove next folklore result about non-existence of limit cycles and the center problem  for planar separable vector fields that we will use in several parts of the paper.

\begin{proposition}\label{pr:sep} 
	The $\mathcal{C}^1$ planar differential equation
	\begin{equation}\label{eq:sep}
		\dot x=f(x)g(y),\quad \dot y=h(x)k(y),
	\end{equation} does not have limit cycles. Moreover, all its monodromic singularities are centers.
\end{proposition}

\begin{proof}  Given any real $\mathcal{C}^1$ function $f$, we denote as $\mathcal{Z}(f)$ the set of all its real zeros. Notice that if $x^*\in\mathcal{Z}(f)$ then $x=x^*$ is an invariant straight line of \eqref{eq:sep}. Similarly happens with the lines $y=y^*$, where $y^*\in\mathcal{Z}(k).$ Then, the greed $$\mathcal{G}=\big(\cup_{x^*\in\mathcal{Z}(f)}\{x=x^*\}\big)\bigcup\left(\cup_{y^*\in\mathcal{Z}(k)}\{y=y^*\}\right)$$ 
	is invariant, and also each of the rectangles $R$ of $\mathbb{R}^2\setminus\mathcal{G},$ given by each of its connected components,  is invariant as well. By the uniqueness of solutions theorem, if~\eqref{eq:sep} has a periodic orbit it must exists an open rectangle $R\subset \mathbb{R}^2\setminus\mathcal{G}$ such that $\gamma\subset R.$ Moreover on $R$ the integrating factor $1/\big({f(x)k(y)}\big)$ is well defined and the resulting system
	\begin{equation*}
		\dot x=\frac{g(y)}{k(y)},\quad \dot y=\frac{h(x)}{f(x)},
	\end{equation*} is $\mathcal{C}^1$  and Hamiltonian. In particular the flow of this new systems preserves the area and this fact prevents the existence of limit cycles. In other words, the periodic orbit $\gamma$ belongs to continua of periodic orbits.
	
	Let us assume that $p$ is a monodromic singularity. As in the above reasoning we know that there exits  $R\subset \mathbb{R}^2\setminus\mathcal{G}$ such that $p$ is in the interior of $R.$ Since on this set the systems has a smooth first integral and the point has a return map it must be a center.
\end{proof}	

Observe that~\eqref{eq:sep} is called  separable vector field because it
also can be written as the separable differential equation
\[
\frac{dy}{dx}=\frac{h(x)}{f(x)}\frac{k(y)}{g(y)}.
\]

\subsection{Resultants}\label{Sec3.5}

Let $f$, $g\colon\mathbb{R}\to\mathbb{R}$ be real polynomials given by,
	\[f(x)=c_0x^l+c_1x^{l-1}+\dots+c_{l-1}x+c_l, \quad g(x)=d_0x^m+d_1x^{m-1}+\dots+d_{m-1}x+d_m,\]
with $c_0d_0\neq0$, $l\geqslant0$ and $d\geqslant0$. The \emph{Sylvester matrix} of $f$ and $g$, in relation to the variable $x$, is the $(l+m)\times(l+m)$ matrix given by,
	\[\operatorname{Syl}(f,g,x)=\underbrace{\left(\begin{array}{cccc}
							c_0 & & & \\
							c_1 & c_0 & & \\
							c_2 & c_1 & \ddots &\\
							\vdots & & \ddots & c_0 \\
							& \vdots & & c_1 \\
							c_l & & & \\
							& c_l & & \vdots \\
							& & \ddots & \\
							& & & c_l 
						\end{array}\right.}_{\text{$m$ columns}}\underbrace{\left.\begin{array}{cccc}
							d_0 & & & \\
							d_1 & d_0 & & \\
							d_2 & d_1 & \ddots & \\
							\vdots & & \ddots & d_0 \\
							& \vdots & & d_1 \\
							d_m & & & \\
							& d_m & & \vdots \\
							& & \ddots & \\
							& & & d_m
						\end{array}\right)}_{\text{$l$ columns}},\]
where the empty spaces are filled by zeros. If $l>0$ and $m>0$, the \emph{resultant} of $f$ and $g$, in relation to the variable $x$, is given by
	\[\operatorname{Res}(f,g,x)=\det\operatorname{Syl}(f,g,x).\]
If $l=0$ or $m=0$, then the resultant is defined as,
\[
\operatorname{Res}(f,g,x)=\begin{cases}
	c_0^m, &\text{ if } l=0,\;m>0;\\
	d_0^l, &\text{ if } l>0,\;m=0;\\
	1, &\text{ if } l=0,\;m=0.
	\end{cases}
	\]

\begin{proposition}[Proposition 3, p. 163 of \cite{CoxLitShe2015}]\label{P1}
	Let $f$, $g\in\mathbb{R}[x]$. Then there is $x_0\in\mathbb{C}$ such that $f(x_0)=g(x_0)=0$ if, and only if, $\operatorname{Res}(f,g,x)=0$. 
\end{proposition}

Let $f$, $g\colon\mathbb{R}\times\mathbb{R}^n\to\mathbb{R}$ be polynomials given by,
\begin{equation}\label{8}
\begin{array}{l}
	f(x,y)=c_0(y)x^l+c_1(y)x^{l-1}+\dots+c_{l-1}(y)x+c_l(y), \vspace{0.2cm} \\
	g(x,y)=d_0(y)x^m+d_1(y)x^{m-1}+\dots+d_{m-1}(y)x+d_m(y),
\end{array}
\end{equation}
where $x\in\mathbb{R}$, $y\in\mathbb{R}^n$, $c_i$, $d_j\colon\mathbb{R}^n\to\mathbb{R}$ are polynomials, $c_0(y)\not\equiv0,$ $d_0(y)\not\equiv0$, $l\geqslant0$ and $d\geqslant0$. Similarly to the case of one variable, the resultant of $f$ and $g$, in relation to the variable $x$, is defined as,
	\[\operatorname{Res}(f,g,x)=\det\operatorname{Syl}(f,g,x).\]
However, instead of a real number, this time $\operatorname{Res}(f,g,x)$ is a polynomial in the variable $y\in\mathbb{R}^n$, i.e. $\operatorname{Res}(f,g,x)\in\mathbb{R}[y]$. To simplify the notation, let $R(y)=\operatorname{Res}(f,g,x)(y)$. Given $y_0\in\mathbb{R}^n$, let $f_{y_0}$, $g_{y_0}\colon\mathbb{R}\to\mathbb{R}$ be given by $f_{y_0}(x)=f(x;y_0)$, $g_{y_0}(x)=g(x;y_0)$ and $R_{y_0}=\operatorname{Res}(f_{y_0},g_{y_0},x)$. Observe that $R_{y_0}\in\mathbb{R}$. In general, $R(y_0)$ and $R_{y_0}$ need not to be equal. However, under some general hypothesis, they do hold a relation.

\begin{proposition}[Proposition 6, p. 167 of \cite{CoxLitShe2015}]\label{P2}
	Let $f$, $g\in\mathbb{R}[x,y]$, $R(y)\in\mathbb{R}[y]$, $f_{y_0}$, $g_{y_0}\in\mathbb{R}[x]$ and $R_{y_0}\in\mathbb{R}$ be as above. If $f_{y_0}$ has degree $l$ and $g_{y_0}$ is nonzero and has degree $p\leqslant m$, then
		\[R(y_0)=c_0(y_0)^{m-p}R_{y_0}.\]
	In particular, $R(y_0)=0$ if, and only if, $R_{y_0}=0$.
\end{proposition}

The next proposition, which follows as a corollary of Proposition~\ref{P2}, shows how the resultants can help to solve a nonlinear system of equations.

\begin{proposition}[Corollary 7, p. 168 of \cite{CoxLitShe2015}]\label{P3}
	Let $f$, $g\colon\mathbb{R}\times\mathbb{R}^n\to\mathbb{R}$ be polynomials given by \eqref{8}. Assume that $y_0\in\mathbb{R}^n$ satisfies the following statements.
	\begin{enumerate}[label=(\alph*)]
		\item $f(x;y_0)\in\mathbb{R}[x]$ has degree $l$ or $g(x;y_0)\in\mathbb{R}[x]$ has degree $m$;
		\item $\operatorname{Res}(f,g,x)(y_0)=0$.
	\end{enumerate}
	Then there is $x_0\in\mathbb{C}$ such that $f(x_0,y_0)=g(x_0,y_0)=0$.
\end{proposition}

We now explain how Propositions~\ref{P1}, \ref{P2} and \ref{P3} can be used to solve a nonlinear system of $n$ equations and $n$ variables. First, for the sake of simplicity, let $n=2$ and consider the nonlinear system,
\begin{equation}\label{9}
	\left\{\begin{array}{l}
		f(x,y)=0, \vspace{0.2cm} \\
		g(x,y)=0,
	\end{array}\right.
\end{equation}
with $(x,y)\in\mathbb{R}^2$. Let $R(y)=\operatorname{Res}(f,g,x)(y)$. Although in general $R$ may have a higher degree than $f$ and $g$, it has only one variable and thus it is possible to study the number and position of all the real solutions of $R(y)=0$.

 Let $y_0\in\mathbb{R}$ be one of such solutions. It follows from Proposition~\ref{P3} that if $f(x;y_0)\in\mathbb{R}[x]$ or $g(x;y_0)\in\mathbb{R}[x]$ maintain its degree, then there is $x_0\in\mathbb{C}$ such that $(x_0,y_0)$ is a solution of \eqref{9}. At this stage, $x_0$ can be obtained by searching all the common real roots of $f(x;y_0)$ and $g(x;y_0)$. In general, if we repeat this process throughout all the real solutions of $R(y)=0$, we may control most, if not all, the real solutions of \eqref{9}. Reciprocally, let $(x_0,y_0)\in\mathbb{R}^2$ be a solution of \eqref{9} and let $f$ and $g$ be given by \eqref{8}. If $c_0(y_0)\neq0$, then it follows from Proposition~\ref{P2} that $R(y_0)=0$ if, and only if, $R_{y_0}=0$. However, since $(x_0,y_0)$ is a solution of \eqref{9}, it follows that $x_0$ is a common root between $f_{y_0}$ and $g_{y_0}$. Therefore, $R_{y_0}=0$ and thus $R(y_0)=0$. That is, if $c_0(y_0)\neq0$, then $y_0$ is a real solution of $R(y)=0$ and thus the solution $(x_0,y_0)$ can be obtained by the method described previously. Let,
\[\begin{array}{l}
	\mathcal{S}=\{(x,y)\in\mathbb{R}^2\colon f(x,y)=g(x,y)=0\}, \vspace{0.2cm} \\
	\mathcal{S}_1=\{(x,y)\in\mathbb{R}^2\colon f(x,y)=g(x,y)=0,\;c_0(y)\neq0\}, \vspace{0.2cm} \\
	\mathcal{S}_2=\{(x,y)\in\mathbb{R}^2\colon f(x,y)=g(x,y)=c_0(y)=0\}.
\end{array}\]
Observe that $\mathcal{S}=\mathcal{S}_1\cup\mathcal{S}_2$, with the union being disjoint. It follows that $\mathcal{S}_1$ is the subset of solutions that can be obtained by the above method. If $c_0(y)\equiv c$, for some $c\in\mathbb{R}\backslash\{0\}$, then $\mathcal{S}_2=\emptyset$ and thus the solutions obtained in $\mathcal{S}_1$ are all the solutions of \eqref{9}. Moreover, even if $c_0(y)$ is not constant, the restriction $c_0(y)=0$ on \eqref{9} may be enough to simplify the system to a point that it can be solved in a simpler way and thus a characterization of $\mathcal{S}_2$ may be achievable. 

The general case of a nonlinear system of $n$ equations and variables follows by using the resultants to reduce it to a system of $n-1$ equations and variables. More precisely, let  $f_j(x_1,\dots,x_n)=0,$ $j=1,2,\ldots n,$
be a system of $n$ equations and variables. By taking the resultants $r_i=\operatorname{Res}(f_1,f_i,x_1)$, $i\in\{2,\dots,n\}$, we reduce the initial system to the new  system $r_i(x_2,\dots,x_n)=0,$ $i=2,\ldots n,$ of $n-1$ equations and variables. Hence, the solutions of this new system can be used to obtain the solutions of initial one. In simple words, theoretically given a system of $n$ system and equations, we take the resultants of those equations and then the resultants of the resultants and so on, until we obtain a system of one equation and one variable, which can be treated relatively easily. We observe that this method may not provide all the solutions of the original system, specially if one does not have enough computational power to calculate all the resultants. For more details about resultants, see \cite[Chapter 3, $\mathsection6$]{CoxLitShe2015} and \cite{GelKapZel1994}.

\section{Finding conditions for a center at the origin}\label{Sec4}

Let $X$ be our family of planar polynomial vector field with a linear center at the origin. As stated in Section~\ref{Sec3.3}, our first approach to prove that some of them have centers is to test their reversibility. If the results are inconclusive, then we can continue by looking by using the Darboux Theory of Integrability to seek for integrating factors by taking as invariant curves the invariant straight lines. We apply this approach to find centers in two subfamilies of $\mathfrak{X}_2^0$ and $\mathfrak{X}_2.$ 

\begin{proposition}\label{P4} Consider the family of vector fields in $\mathfrak{X}_2^0,$ 
	$X=(P,Q)$ where
\begin{equation*}
	\begin{array}{l}
		P(x,y)=(4x^2-1)(y+a_{11}xy+a_{02}y^2), \vspace{0.2cm} \\ Q(x,y)=(4y^2-1)(-x+b_{20}x^2+b_{11}xy),
	\end{array}
\end{equation*}
with $a_{ij}$, $b_{ij}\in\mathbb{R}$ and a linear center at the origin. Then the origin is a center if, and only if, one of the following conditions is satisfied:
\begin{align*}
 &(C_1)\,\,	a_{11}=b_{20}=0; &&  (C_2)\,\, a_{11}=b_{11}=0; && (C_3)\,\, a_{02}=b_{20}=0;\\
  &(C_4)\,\,	a_{02}=b_{11}=0; &&  (C_5)\,\, b_{20}=a_{02},\,b_{11}=a_{11};&&\\ &(C_6)\,\,  b_{20}=-a_{02},\, b_{11}=-a_{11}; &&(C_7)\,\, b_{20}=-a_{11},\, b_{11}=-a_{02}.
	\end{align*}	
\end{proposition}

\begin{proof} 
 We start by proving that these conditions are sufficient for the origin to be a center. The proof follows by a case-by-case study.	First we test the reversibility of each family in relation to the straight lines $x=0$, $y=0$, $y=x$ and $y=-x$ (see Section~\ref{Sec3.3}). With this approach, we conclude that families $C_1$, $C_4$, $C_5$ and $C_6$ are reversible in relation to the lines $x=0$, $y=0$ $y=-x$ and $y=x$, respectively.
	
For those that are not reversible, we use the results of Section~\ref{Sec3.4} and~~\ref{Sec3.4b}. With the Darboux Theory of Integrability (Section~\ref{Sec3.4}) 	we get that  families $C_2$ and $C_7$ have integrating factors $R_2$ and $R_7$ given by,
	\[R_2(x,y)=\frac{1}{(4x^2-1)(4y^2-1)}, \quad R_7(x,y)=\frac{1}{a_{11}x+a_{02}y+1}.\]
Observe that for any choice of the parameters, there is a neighborhood of the origin such that $R_2$ and $R_7$ are well defined and so the origin is a center. Finally, family $C_3$ is under the hypotheses  of Proposition~\ref{pr:sep} and the origin is again a center.
	
We now proof that these conditions are all the possible center conditions. Since we have four free parameters, namely $(a_{11},a_{02},b_{20},b_{11})$, it is natural to calculate the first four Lyapunov quantities of $X$ at the origin. The first two Lyapunov quantities are given by,
\begin{equation}\label{13}
\begin{array}{l}
	\displaystyle L_1=\frac{2}{3}a_{02}a_{11}-\frac{2}{3}b_{11}b_{20}, \vspace{0.2cm} \\
	\displaystyle L_2=-\frac{20}{9}b_{20}^2a_{02}a_{11}+\frac{14}{9}b_{20}^3b_{11}+\frac{2}{3}a_{11}a_{02}^3+\frac{2}{5}b_{11}a_{11}a_{02}^2-\frac{14}{15}a_{02}a_{11}^3-\frac{26}{9}b_{20}a_{02}a_{11}^2+ \vspace{0.2cm} \\
	\displaystyle \qquad\qquad+\frac{2}{15}b_{11}^2b_{20}a_{02}+\frac{16}{15}b_{11}b_{20}a_{11}^2+\frac{106}{45}b_{11}b_{20}^2a_{11}-\frac{2}{15}b_{20}b_{11}^3-\frac{32}{15}a_{02}a_{11}+\frac{32}{15}b_{11}b_{20}.
\end{array}
\end{equation}
The constants $L_3$ and $L_4$ are polynomials of degree $6$ and $8$, with $44$ and $112$ monomials, respectively. Hence, for the sake of their size, we will not write them here. We now use the theory of resultants (see Section~\ref{Sec3.5}) to obtain the solutions of the nonlinear system,
\begin{equation}\label{14}
\left\{\begin{array}{l}
	L_1(a_{11},a_{02},b_{20},b_{11})=0, \vspace{0.2cm} \\
	L_2(a_{11},a_{02},b_{20},b_{11})=0, \vspace{0.2cm} \\
	L_3(a_{11},a_{02},b_{20},b_{11})=0, \vspace{0.2cm} \\
	L_4(a_{11},a_{02},b_{20},b_{11})=0.
\end{array}\right.
\end{equation}
First, we calculate the resultants,
	\[r_2=\operatorname{Res}(L_1,L_2,a_{11}), \quad r_3=\operatorname{Res}(L_1,L_3,a_{11}), \quad r_4=\operatorname{Res}(L_1,L_4,a_{11}).\]
They are polynomials in $(a_{02},b_{20},b_{11})$ and can be factored as,
	\[r_2=\mathcal{R}_1\mathcal{R}_2\mathcal{R}_3\mathcal{R}_4\mathcal{R}_5\mathcal{R}_6\mathcal{R}_{12}, \quad r_3=\mathcal{R}_1\mathcal{R}_2\mathcal{R}_3\mathcal{R}_4\mathcal{R}_5\mathcal{R}_6\mathcal{R}_{13}, \quad r_4=\mathcal{R}_1\mathcal{R}_2\mathcal{R}_3\mathcal{R}_4\mathcal{R}_5\mathcal{R}_6\mathcal{R}_{14},\]
where
	\[\mathcal{R}_1=b_{11}, \quad \mathcal{R}_2=b_{20}, \quad \mathcal{R}_3=a_{02}, \quad \mathcal{R}_4=a_{02}-b_{20}, \quad \mathcal{R}_5=a_{02}+b_{11}, \quad \mathcal{R}_6=a_{02}+b_{20}.\]
and $\mathcal{R}_{12}=\frac{16}{81}a_{02}-\frac{16}{405}b_{11}$. The polynomials $\mathcal{R}_{13}$ and $\mathcal{R}_{14}$ are of degree $5$ and $9$ and have $9$ and $28$ monomials, respectively. Following the notation of Section~\ref{Sec3.5}, we write $x=a_{11}\in\mathbb{R}$, $y=(a_{02},b_{20},b_{11})\in\mathbb{R}^3$ and thus it follows from \eqref{13} that,
	\[L_1(x,y)=c_0(y)x+c_1(y),\]
where $c_0(y)=\frac{2}{3}a_{02}$ and $c_1(y)=-\frac{2}{3}b_{11}b_{20}$. Hence, $c_0(y)=0$ if, and only if, $a_{02}=0$. Therefore, it follows from Section~\ref{Sec3.5} that if $(a_{11}^*,a_{02}^*,b_{20}^*,b_{11}^*)$ is a solution of \eqref{14}, with $a_{02}^*\neq0$, then $(a_{02}^*,b_{20}^*,b_{11}^*)$ is either a solution of $\mathcal{R}_i=0$, for some $i\in\{1,\dots,6\}$, or it is a solution of
\begin{equation}\label{15}
	\left\{\begin{array}{l}
		\mathcal{R}_{12}(a_{02},b_{20},b_{11})=0, \vspace{0.2cm} \\
		\mathcal{R}_{13}(a_{02},b_{20},b_{11})=0, \vspace{0.2cm} \\
		\mathcal{R}_{14}(a_{02},b_{20},b_{11})=0.
	\end{array}\right.
\end{equation}
We now go to the second round of resultants. That is, let
	\[R_3=\operatorname{Res}(\mathcal{R}_{12},\mathcal{R}_{13},a_{02}), \quad R_4=\operatorname{Res}(\mathcal{R}_{12},\mathcal{R}_{14},a_{02}).\]
They are polynomials in $(b_{20},b_{11})$ and can be factored as,
	\[R_3=\mathcal{R}_1^3\mathcal{R}_{23}, \quad R_4=\mathcal{R}_1^5\mathcal{R}_{24},\]
with $\mathcal{R}_{23}$ and $\mathcal{R}_{24}$ polynomials in $(b_{20},b_{11})$. Since
\[	\mathcal{R}_{12}=\frac{16}{81}a_{02}-\frac{16}{405}b_{11} \]
 has constant coefficient in $a_{02}$, it follows from Section~\ref{Sec3.5} that if $(a_{02}^*,b_{20}^*,b_{11}^*)$ is a solution of \eqref{15}, then $(b_{20}^*,b_{11}^*)$ is either a solution of $\mathcal{R}_1=0$, or it is a solution of
\begin{equation}\label{16}
	\left\{\begin{array}{l}
		\mathcal{R}_{23}(b_{20},b_{11})=0, \vspace{0.2cm} \\
		\mathcal{R}_{24}(b_{20},b_{11})=0. 
	\end{array}\right.
\end{equation}
Since 
	$\mathcal{R}_{23}=\alpha_1b_{11}^2+\alpha_2b_{20}^2+\alpha_3,$
with $\alpha_i\in\mathbb{R}\backslash\{0\}$, has constant coefficient in $b_{20}$ and $\operatorname{Res}(\mathcal{R}_{23},\mathcal{R}_{24},b_{20})= c\in\mathbb{R}\backslash\{0\}$, it follows that \eqref{16} has no solution. Therefore, we conclude that if we let
\[\begin{array}{l}
	\displaystyle \mathcal{S}=\{(a_{11},a_{02},b_{20},b_{11})\in\mathbb{R}^4\colon L_1=L_2=L_3=L_4=0\}, \vspace{0.2cm} \\
	\displaystyle \mathcal{S}_k=\{(a_{11},a_{02},b_{20},b_{11})\in\mathbb{R}^4\colon L_1=L_2=L_3=L_4=\mathcal{R}_k=0\}, \vspace{0.2cm} \\
	\mathcal{S}^*=\{(a_{11},a_{02},b_{20},b_{11})\in\mathbb{R}^4\colon L_1=L_2=L_3=L_4=a_{02}=0\},
\end{array}\]
with $k\in\{1,\dots,6\}$, then $\mathcal{S}=\mathcal{S}^*\cup\mathcal{S}_1\cup\dots\cup\mathcal{S}_6$.
Since $\mathcal{R}_3=a_{02}$, it follows that $\mathcal{S}^*=\mathcal{S}_3$ and thus we conclude that $\mathcal{S}=\mathcal{S}_1\cup\dots\cup\mathcal{S}_6$. Since the polynomials $\mathcal{R}_i$ are quite simple, it is now easy to obtain conditions $C_1,\dots,C_7$ mentioned on the enunciate of the proposition. For example, if $\mathcal{R}_1=b_{11}=0$, then it follows from \eqref{13} that $L_1=0$ if, and only if, $a_{11}=0$ or $a_{02}=0$. Hence, we obtain two possible set of conditions, given by $\{b_{11}=a_{11}=0\}$ and $\{b_{11}=a_{02}=0\}$. Simple calculations show that both these families are solutions of \eqref{14} and thus they are candidates for center conditions (and indeed, they are precisely equal to $C_2$ and $C_4$, respectively). This finishes the proof. \end{proof}

\begin{proposition}\label{P5} Consider the family of vector fields in $\mathfrak{X}_2,$ 
$X=(P,Q)$ where
\begin{equation*}
	\begin{array}{l}
		P(x,y)=(4x^2-1)(y+a_{20}x^2+a_{11}xy+a_{02}y^2), \vspace{0.2cm} \\ Q(x,y)=(4y^2-1)(-x+b_{20}x^2+b_{11}xy+b_{02}y^2),
	\end{array}
\end{equation*}
with $a_{ij}$, $b_{ij}\in\mathbb{R}$ and a linear center at the origin. If one of the conditions $(D_j)$ holds, where
\begin{align*}
&(D_1)\,\,   a_{20}=a_{02}=b_{20}=b_{02}=0; &&  (D_2)\,\,    a_{20}=a_{11}=b_{11}=b_{02}=0;\\
&(D_3)\,\,    a_{20}=a_{02}=b_{11}=0; &&  (D_4)\,\,    a_{11}=b_{20}=b_{02}=0;\\
&(D_5)\,\,    b_{20}=a_{02},\, b_{11}=a_{11},\, b_{02}=a_{20}; &&  (D_6)\,\,    b_{20}=-a_{02},\, b_{11}=-a_{11},\, b_{02}=-a_{20};\\
&(D_7)\,\,    a_{20}=a_{02}=-b_{11}/2,\, b_{20}=b_{02}=-a_{11}/2; &&  (D_8)\,\,    a_{11}=2a_{20}=-2b_{02}=-b_{11},\, b_{20}=-a_{02};\\
&(D_9)\,\,    a_{11}=-2a_{20}=-2b_{02}=b_{11},\, b_{20}=a_{02}; &&  	
\end{align*}
then the origin is a center.
\end{proposition}

\begin{proof} Similarly that in  Proposition~\ref{P4}, here the proof follows by a case-by-case study. First we observe that $D_4$ is reversible in relation to $x=0$, $D_3$ is reversible in relation to $y=0$, $D_6$ and $D_8$ are reversible in relation to $y=x$ and $D_5$ and $D_9$ are reversible in relation to $y=-x$. On the other hand, $D_2$ and $D_7$ have integrating factors $R_2$ and $R_7$ given by,
	\[R_2(x,y)=\frac{1}{(4x^2-1)(4y^2-1)}, \quad R_7(x,y)=\frac{1}{(4x^2-1)^2(4y^2-1)^2}.\]
Finally, we observe that family $D_1$ is precisely equal to family $C_3$ in Proposition~\ref{P4} and thus is also a center. This finishes the proof. \end{proof}

We now make some comments about the process to find the center condition $D_1,\dots,D_9$ and explain why we cannot ensure that these are all the possible conditions, as we did in Proposition~\ref{P4}. First, since we have six free parameters, we must calculate the first six Lyapunov quantities. For the sake of their sizes, we will not write then here. The constant $L_6$ for example, has $5425$ monomials. Let $r_i=\operatorname{Res}(L_1,L_i,\lambda)$, where $\lambda\in\{a_{20},a_{11},a_{02},b_{20},b_{11},b_{02}\}$ and $i\in\{2,\dots,6\}$. If $\lambda\in\{a_{20},b_{02}\}$, then $r_1,\dots,r_6$ show no common factor. However, if $\lambda\in\{a_{11},a_{02},b_{20},b_{02}\}$, then $r_1,\dots,r_6$ have a common factor. More precisely, let
	\[\mathcal{R}_1=a_{20}+a_{02}, \quad \mathcal{R}_2=a_{11}+2b_{02}, \quad \mathcal{R}_3=2a_{20}+b_{11}, \quad \mathcal{R}_4=b_{20}+b_{02}.\]
If $\lambda=a_{11}$, (resp. $\lambda=a_{02}$, $b_{20}$, $b_{02}$), then $\mathcal{R}_1$ (resp. $\mathcal{R}_2$, $\mathcal{R}_2$, $\mathcal{R}_4$) is a common factor of $r_1,\dots,r_6$. Therefore, we are motivated to seek for center conditions among the sets,
	\[\mathcal{S}_k=\{(a_{20},\dots,b_{02})\in\mathbb{R}^6\colon L_1=\dots=L_6=\mathcal{R}_k=0,\}\]
for $k\in\{1,\dots,4\}$. As in the proof of Proposition~\ref{P4}, since the polynomials $\mathcal{R}_k$ are quite simple, it is now easy to deduce conditions $D_1,\dots,D_9$. Let $r_i=\operatorname{Res}(L_1,L_i,a_{11})$, $i\in\{2,\dots,6\}$. As already mentioned, we have $r_i=\mathcal{R}_1\mathcal{R}_{1,i}$, $i\in\{2,\dots,6\}$. If we let $R_j=\operatorname{Res}(\mathcal{R}_{11},\mathcal{R}_{12},a_{20})$, $j\in\{3,\dots,6\}$, then we obtain that
	\[R_3=\mathcal{R}_4^9\mathcal{R}_7^3\mathcal{R}_8\mathcal{R}_9^8\mathcal{R}_{10}\mathcal{R}_{11}\mathcal{R}_{23}, \quad R_4=\mathcal{R}_4^{13}\mathcal{R}_7^3\mathcal{R}_8\mathcal{R}_9^{12}\mathcal{R}_{10}\mathcal{R}_{11}\mathcal{R}_{24},\]
where,
	\[\mathcal{R}_7=b_{02}, \quad \mathcal{R}_8=a_{02}-b_{20}, \quad \mathcal{R}_9=2a_{02}-b_{11}, \quad \mathcal{R}_{10}=a_{02}+b_{20}, \quad \mathcal{R}_{11}=a_{02}-b_{11}.\]
Moreover, $\mathcal{R}_{23}$ and $\mathcal{R}_{24}$ have degrees $28$ and $44$ and have $5774$ and $34122$ monomials, respectively. At this stage we lack computational power to keep the computations and to seek for a factorization of $R_5$ and $R_6$. If we calculate all the possible combinations of resultants (i.e. if we take the resultants in relation to $\lambda$ and then in relation to $\mu$, for any $\lambda\in\{a_{20},\dots,b_{02}\}$ and then for any $\mu\in\{a_{20},\dots,b_{02}\}\backslash\{\lambda\}$), as far as our computational power allows, we obtain twenty five other polynomials $\mathcal{R}_i$, in addition to the above eleven. With these polynomials in hand, we studied the sets
	\[\mathcal{S}_k=\{(a_{20},\dots,b_{02})\in\mathbb{R}^6\colon L_1=\dots=L_6=\mathcal{R}_k=0,\}\]
for $k\in\{5,\dots,36\}$, seeking for new center conditions. However, no new center condition were founded. Since we did not calculated every necessary resultant, we cannot ensure that families $D_1,\dots,D_9$ are all the possible solutions of $L_1=\dots=L_6=0$. Finally, we observe that the use of the command \emph{solve} (and other similar commands) in well-known symbolic and numeric softwares (such as \emph{Mathematica} and \emph{Maple}) results in the return of no solutions at all or in only six of the nine solutions that we have founded. 

\section{Bifurcation of limit cycles from the origin}\label{Sec5}

In this section we apply Theorem~\ref{GGT} on the center families $C_1,\dots,C_7$ and $D_1,\dots, D_9$ of Propositions~\ref{P4} and \ref{P5} to obtain a lower bound for the number of limit cycles in $\mathfrak{X}_2^0$ and $\mathfrak{X}_2$. For the sake of simplicity, we only show the demonstration of the families that provided the most number of limit cycles.

\begin{proposition}\label{P7}
	Let $X=(P,Q)\in\mathfrak{X}_2^0$ be the vector field
		\[P(x,y)=(4x^2-1)(y+\mu xy), \quad Q(x,y)=(4y^2-1)(-x+x^2),\]
	with $\mu\in\mathbb{R},$ which satisfies condition~$C_4$ in Proposition~\ref{P4} and has a center at the origin. Then, for $\mu=5$, there are perturbations of $X$ in $\mathfrak{X}_2^0$ bifurcating at least three limit cycles from the origin.
\end{proposition}

\begin{proof}  Let $Y=(R,S)$ be the perturbation of $X$ given by,
	\[R(x,y)=P(x,y)+(4x^2-1)(-\alpha y+ a y^2), \quad S(x,y)=Q(x,y)+(4y^2-1)(\alpha x+ b xy),\]
with $\alpha,  a ,  b \in\mathbb{R}$. When $\alpha=0,$ by computing the linear terms of the first four Lyapunov quantities in relation to $\Lambda=\{ a , b \}$, we observe that the rank of the jacobian matrix in relation to $\Lambda$ is two (i.e. under the notation of Theorem~\ref{GGT}, $k=2$). Due to size, we will only show the first two Lyapunov quantities:
\[\begin{array}{l}
	\displaystyle L_1=\frac{2}{3}\mu  a -\frac{2}{3} b +O_2( a , b ), \vspace{0.2cm} \\
	\displaystyle L_2=-\left(\frac{14}{15}\mu^3+\frac{26}{9}\mu^2+\frac{196}{45}\mu\right) a +\left(\frac{166}{45}+\frac{16}{15}\mu^2+\frac{106}{45}\mu\right) b +O_2( a , b ).
\end{array}\] 
Taking the first Lyapunov quantity in relation to $\Lambda$ we have that the rank of the jacobian matrix is one. Hence, when $\mu \ne0,$ after the change of parameters 
\[
\left(\begin{matrix} a \\ b \end{matrix}\right)= M \left(\begin{matrix}\lambda_1\\ \lambda_2 \end{matrix}\right),\quad \mbox{with}\quad M=\left(\begin{matrix}\frac{3}{2\mu}& \frac1{\mu}\\0&1\end{matrix}\right),
\]
 we can write,
\[\begin{array}{l}
	\displaystyle L_1=\lambda_1+O_2(\lambda), \vspace{0.2cm} \\
	\displaystyle L_2=g_{21}(\mu)\lambda_1+f_0(\mu) \lambda_2 +O_2(\lambda), \vspace{0.2cm} \\
	L_3=g_{31}(\mu)\lambda_1+f_1(\mu) \lambda_2 +O_2(\lambda),
\end{array}\] 
where $g_{21}$ and $g_{31}$ are polynomials of degree $2$ and $4$, respectively, $\lambda=(\lambda_1,\lambda_2)$ and
	\[f_0(\mu)=\frac{2}{15}(\mu+1)(\mu-5), \quad f_1(\mu)=-\frac{2}{315}(\mu+1)(47\mu^3-96\mu^2-582\mu-673).\]
Since $\mu=5$ is a simple zero of $f_0$ and $f_1(5)\neq0$, it follows from Theorem~\ref{GGT} with $\ell=1$ that we can bifurcate at least three limit cycles from the origin. \end{proof}

\begin{remark}
	We remark that $C_4$ is not the only family from which we can bifurcate three limit cycles. Families $C_1$, $C_5$ and $C_6$ also achieved this number. Moreover, under the notations of Theorem~\ref{GGT}, for each of these families we have $k=2$ and $\ell=1$.
\end{remark}

\begin{proposition}\label{P8}
	
	Let $X=(P,Q)\in\mathfrak{X}_2$ be the vector field
		\[\begin{array}{l}
		\displaystyle P(x,y)=(4x^2-1)\left(y-x^2+ \mu xy-y^2\right), \vspace{0.2cm} \\ \displaystyle Q(x,y)=(4y^2-1)\left(-x-\mu x^2/2+2xy-\mu y^2/2\right),
	\end{array}\]
	with $ \mu \in\mathbb{R},$ which satisfies condition~$D_7$ in Proposition~\ref{P5} and has a center at the origin. 
	 Set $ \mu =\mu^\pm$, where $\mu^\pm$ is one of the two real zeros of
	$h( \mu )=175 \mu ^4+128520 \mu ^2-44944.$
	Then, there are perturbations of $X$ in $\mathfrak{X}_2$ bifurcating at least five limit cycles from the origin.
\end{proposition}

\begin{proof} Let $Y=(R,S)$ be the perturbation of $X$ given by,
	\[\begin{array}{l}
		\displaystyle R(x,y)=P(x,y)+(4x^2-1)(-\alpha y+a_1xy+a_2y^2), \vspace{0.2cm} \\ \displaystyle S(x,y)=Q(x,y)+(4y^2-1)(\alpha x+b_1x^2+b_2xy),
	\end{array}\]
with $\alpha,$ $a_1$, $a_2$, $b_1$, $b_2\in\mathbb{R}$. When $\alpha=0,$ by computing the linear terms of the first six Lyapunov quantities in relation to $\Lambda=\{a_1,a_2,b_1,b_2\}$, we observe that the rank of the jacobian matrix in relation to $\Lambda$ is four (i.e. under the notation of Theorem~\ref{GGT}, $k=4$). Due to size, we will only show the first two Lyapunov quantities:
	\[\begin{array}{l}
		\displaystyle L_1=-\frac{4}{3}a_1+\frac{2}{3} \mu b_2+O_2(a_1,a_2,b_1,b_2), \vspace{0.2cm} \\
		\displaystyle L_2=\left(\frac{2}{9} \mu ^2+\frac{88}{15}\right)a_1+\frac{32}{15} \mu a_2-\frac{64}{15}b_1-\left(\frac{7}{9} \mu ^3+\frac{4}{15} \mu \right)b_2+O_2(a_1,a_2,b_1,b_2).
	\end{array}\] 
Taking the first three Lyapunov quantities in relation to $\Lambda$ we have that the rank of the jacobian matrix is three. Hence, there is a change in the variables $(a_1,a_2,b_1,b_2)\mapsto(\lambda_1,\lambda_2,\lambda_3,\lambda_4)=\lambda$, with $\lambda_4=b_2,$ that for shortness we do not give explicitly, after which $L_i=\lambda_i+O_2(\lambda),$ $i=1,2,3,$ and
\[		\displaystyle L_4=\left(\sum_{i=1}^{3}g_{4,i}( \mu )\lambda_i\right)+f_0( \mu )\lambda_4+O_2(\lambda), \quad
		\displaystyle L_5=\left(\sum_{i=1}^{3}g_{5,i}( \mu )\lambda_i\right)+f_1( \mu )\lambda_4+O_2(\lambda),\] 
where,
	\[\begin{array}{l}
		\displaystyle f_0( \mu )=\frac{1}{113400} \mu ( \mu -2)( \mu +2)h( \mu ), \vspace{0.2cm}\\ \displaystyle f_1( \mu )=-\frac{1}{623700} \mu ( \mu -2)( \mu +2)(5775 \mu ^6+468020 \mu ^4+852000 \mu ^2+1150336),
	\end{array}\]
where $h$ is given in the statement. We observe that $h$ has exactly two real roots, $\mu=\mu^\pm\approx\pm0.5912.$  Since each $\mu^\pm$ is a simple zero of $f_0$ and $f_1(\mu^\pm)\neq0$,  it follows from Theorem~\ref{GGT} with $\ell=1$ that by taking $ \mu =\mu^\pm,$  we can bifurcate at least five limit cycles from the origin. \end{proof}

\begin{remark}
	We remark that $D_7$ is not the only family from which we can bifurcate five limit cycles. Families $D_1$, $D_5$ and $D_6$ also achieved this number. Under the notations of Theorem~\ref{GGT}, families $D_7$ and $D_1$ have $k=4$ and $\ell=1$, while families $D_5$ and $D_6$ achieved five limit cycles with  $k=3$ and $\ell=2$.
\end{remark}

\section{Simultaneous bifurcation of limit cycles from two centers}\label{Sec6}

In this section we will take advantage of a symmetry of the family of vector fields $X=(P,Q)$ with
\[
P=(4x^2-1)(a_{00}+a_{20}x^2+a_{11}xy+ a_{02}y^2), \quad Q(x,y)=(4y^2-1)(b_{00}+b_{20}x^2+b_{11}xy+ b_{02}y^2).
\]
Notice that it is invariant by the change of variables and time $(x,y,t)\mapsto(-x,-y,-t).$  Hence if some behaviour appears in the half plane $\mathcal{H}^+=\{x>0\}$ it also appears in $\mathcal{H}^-=\{x<0\},$ but with reversed stabilities. In particular, if a critical point $p\in\mathcal{H}^+$ is surrounded by $N$ limit cycles in this region the vector field has at least $2N$ limit cycles, because there are $N$ symmetric  limit cycles in~$\mathcal{H}^-.$

\begin{proposition}\label{P9}
	Let $X=(P,Q)$ be the family of vector fields
\begin{equation*}
	\begin{array}{l}
		\displaystyle P(x,y)=(4x^2-1)(xy+\lambda y^2), \vspace{0.2cm} \\
		\displaystyle Q(x,y)=(4y^2-1)(1-16x^2-4\alpha xy+\mu y^2),
	\end{array}
\end{equation*}	
with $a,\alpha,\mu\in\mathbb{R}.$ Then:
\begin{itemize}
\item[(a)]  When $\alpha=\lambda=0,$ then, for all values of $\mu,$ $X$ has two reversible centers at the points $p^{\pm}=(\pm 1/4,0).$ 

\item[(b)] When $\alpha=0$ and  $\mu=-{3}/{8}$ there is a perturbation inside this family (and so in $\mathfrak{X}_2$) bifurcating two limit cycles from both $p^\pm,$ simultaneously.

\item[(c)] When $\alpha=\mu=0$ there is a perturbation inside this family and keeping $\mu=0$ (and so in $\mathfrak{X}_2^0$) bifurcating one limit cycle from both $p^\pm,$ simultaneously.
\end{itemize}	
\end{proposition}

\begin{proof} First of all notice that from the comments before the statement of the proposition it suffices to study the point $p^+$ and the bifurcations from it. This is so,  because all the periodic orbits $\gamma$ around it also appear around~$p^-$ as $-\gamma,$ with reversed orientation and stability.

	$(a)$ When $\alpha=\lambda=0$  then $X$ is reversible in relation to the straight line $y=0$. Since  $p^+$  lies on this line it follows that it is a center, see Section~\ref{Sec3.3} for more details.
	
	$(b)$ When $\alpha=0$ let us compute the expansions in $\lambda$ of the Lyapunov quantities for $p^+.$ We get
	\begin{equation*}
		L_1=f_0(\mu)\lambda+O_2(\lambda), \quad  L_2=f_1(\mu)\lambda+O_2(\lambda),
	\end{equation*}
	where 
	\[
	f_0(\mu)=\frac{\sqrt{6}}{9}(8\mu+3),\quad f_1(\mu)=-\frac{\sqrt{6}}{12960}(22528\mu^3+6464\mu^2+27024\mu+1503).\]  By taking $k=1$, $\ell=1$ and $\mu_0=-3/8$ we are under the hypotheses of Theorem~\ref{GGT} and we obtain two limit cycles around $p^+.$
	
	$(c)$ When $\mu=0$ we have a classical Andronov-Hopf bifurcation around $p^+$ and a single limit cycle appears.
	\end{proof}

\section{Proof of the main results}\label{Sec7}

We start proving Proposition~\ref{T1.5}. As we will see it will follow as a corollary of a much general result. We recall some definitions. As usual, we will denote  by $\operatorname{ind}_X(p)$ the index of the singularity $p$ with respect to the vector field $X$ (see \cite[Chapter $6$]{DumLliArt2006}). Recall that simple singularities correspond to the ones such that $\det DX(p)\ne0$ and for them $\operatorname{ind}_X(p)=\operatorname{sgn}(DX(p)).$

Finally, given the set of singularities of a vector field $X$ in~\cite{CGM1993} it is defined what is understood as their {\it configuration}. Without entering in full detail the configuration encodes the position of the singularities together with their indices. In particular, Berlinski\u \i's Theorem stated in the introduction  gives all the possible configurations of the singularities of quadratic vector fields with four singularities.

\begin{lemma}\label{le:det}
 Consider the planar vector fields $X=(uf,vg)$ and $Y=(f,g)$ where
 $f,g, u$ and $v$ are~$\mathcal{C}^1$ functions from $\mathbb{R}^2$ into $\mathbb{R}.$ Let $p$ be a singularity of $Y$ (and so, of $X$). Then \begin{equation}
 	\label{eq:det}
 \det DX(p)=u(p)v(p)\det DY(p).\end{equation}
 In particular, if $u(p)v(p)\det DY(p)\ne0,$ 
 then \begin{equation}
 	\label{eq:ind}\operatorname{ind}_X (p)= \operatorname{sgn}(u(p)v(p))\operatorname{ind}_Y (p).
 \end{equation}
 \end{lemma}

\begin{proof}  Since $f(p)=g(p)=0,$
the Jacobian matrix of $X$ at $p$ is given by,
\begin{equation}\label{39}
	DX(p)=\left(\begin{array}{cc} 
		\displaystyle u(p)\frac{\partial f}{\partial x}(p) & \displaystyle  u(p)\frac{\partial f}{\partial y}(p) \vspace{0.2cm} \\
		\displaystyle  v(p)\frac{\partial g}{\partial x}(p) & \displaystyle  v(p)\frac{\partial g}{\partial y}(p) 
	\end{array}\right)=u(p)v(p)\left(\begin{array}{cc} 
	\displaystyle \frac{\partial f}{\partial x}(p) & \displaystyle  \frac{\partial f}{\partial y}(p) \vspace{0.2cm} \\
	\displaystyle  \frac{\partial g}{\partial x}(p) & \displaystyle  \frac{\partial g}{\partial y}(p) 
	\end{array}\right).
\end{equation}
Therefore, equality~\eqref{eq:det} follows.\end{proof} 

Next result is a straightforward consequence of equality~\eqref{eq:ind}.

\begin{proposition}\label{pr:gen}  Let $Y=(f,g)$ be a~$\mathcal{C}^1$ planar vector field having a given configuration of singularities and with all its singularities simple. If  $u,v:\mathbb{R}^2\to\mathbb{R}$ are functions of class~$\mathcal{C}^1$ and  all the singularities $p$ of $Y$ are in the set\, $\mathfrak{D}:=\{q\in\mathbb{R}^2\,:\, u(q)v(q)>0 \}$ then the configuration of singularities of $X=(uf,vg)$ contained in~$\mathfrak{D}$ coincides with the configuration of the singularities of $Y.$ 
\end{proposition}

Notice that in the above result the saddle points of $Y$ keep as saddle points for $X$ and the same happens with the antisaddles, but center points can be converted into nodes or foci, and reciprocally.  Moreover  nodes and foci can change their stabilities because the signs of the traces of the two matrices $DX(p)$ and $DY(q)$ substantially change. Moreover, obviously, when the singularities are in $\{q\in\mathbb{R}^2\,:\, u(q)v(q)<0 \}$ saddles for $X$ are converted into antisaddles for $Y$ and vice versa.

\begin{proof}[Proof of Proposition~\ref{T1.5}]  
	We will apply Proposition~\ref{pr:gen} with $u(x,y)=4x^2-1$ and $v(x,y)=4y^2-1$ and $Y=(f,g),$ being $f$ and $g$  quadratic polynomials. Notice that if $Y$ has four singularities all of them must be simple and by hypotheses they are in $\Delta.$ Moreover $\Delta\subset \mathfrak{D}.$  Therefore the proposition follows by Proposition~\ref{pr:gen} and by using that  Berlinski\u \i's Theorem holds for the quadratic vector field $Y$.
\end{proof}

Although it is not the aim of this work we give a similar result to Proposition~\ref{T1.5} but for cubic Kolmogorov systems. Its proof is identical to the above one.

\begin{proposition}Assume that $X=(x f(x,y), y g(x,y)),$ where $f$ and $g$ are quadratic polynomials and it has four singularities in $\Delta'=\{(x,y)\in\mathbb{R}^2\,:\, x>0,\, y>0\}.$ Then the conclusions of Berlinski\u \i's Theorem 
	also hold for these four singularities.
	\end{proposition}
	
We can also apply Proposition~\ref{pr:gen} to the vector fields of degree five $X\in\mathfrak{X}_3$ having nine singularities in $\Delta$ or to $Y$  being a quartic Kolmogorov system having nine singularities en $\Delta'.$  Then from the Benlinskii's type results for cubic systems proved in~\cite{CGM1993,LV} we have the corresponding results for~$X$ and~$Y.$

Before we continue to work on the proof of Theorems~\ref{T1} and~\ref{T2}, we recall some more definitions. Let $X$ be a planar polynomial vector field of degree $n$ and $\ell\subset\mathbb{R}^2$ an algebraic curve of degree $m$, i.e. $\ell=F^{-1}(0)$, for some polynomial $F\colon\mathbb{R}^2\to\mathbb{R}$ of degree $m$. We say that $p\in\ell$ is a \emph{contact point} between $X$ and $\ell$ if $X(p)$ is tangent to $\ell$. Let
\[H(x,y)=\left<X(x,y),\nabla F(x,y)\right>,\]
where $\nabla F$ is the gradient of $F$ and $\left<\cdot,\cdot\right>$ is the standard inner product of $\mathbb{R}^2$. Observe that $H$ is a polynomial of degree $n+m-1$ and that $(x,y)\in\mathbb{R}^2$ is a contact point between $X$ and $\ell$ if, and only if, it is a solution of
\[\left\{\begin{array}{l}
	F(x,y)=0, \vspace{0.2cm} \\
	H(x,y)=0.
\end{array}\right.\]
If $\ell$ is not invariant by the flow of $X$, then it follows from Bezout's Theorem that there are at most $m(n+m-1)$ contact points between $X$ and $\ell$. Moreover, observe that if $p\in\ell$ is a singularity of $X$, then it is also a contact point.

\begin{proof}[Proof of Theorem~\ref{T1}] Statements $(b)$ and $(c)$ follows directly from Propositions~\ref{P8} and \ref{P9}\,$(b)$, respectively. Hence, we need only to prove that if $X\in\mathfrak{X}_2$, then $X$ can have at most two centers in $\Delta$. 
	
Firstly we will prove that if $X$ had three centers in $\Delta$ then they would be non-degenerated centers. This follows because centers have index $+1$  and it is known that the index of a singularity is congruent, modulus 2, with the multiplicity of this singularity \cite[Theorems $1.1$\&$2.1$]{EL} (for more details, we also refer to \cite[Section $6.8$]{DumLliArt2006}). In particular, non-degenerated centers must have multiplicity at least 3, and as a consequence if one of such centers exists $X$ would have at most two singularities in $\Delta.$ Hence we can assume that the three centers would be non-degenerated.
	
As we have already explained in the introduction, in our proof we will follow the same scheme of the proof of the same result for quadratic systems, see~\cite{Cop1966}.
	
Let $X=(P,Q)\in\mathfrak{X}_2$ be given by
\begin{equation}\label{37}
	P(x,y)=(4x^2-1)f(x,y), \quad Q(x,y)=(4y^2-1)g(x,y),
\end{equation}
where $f$ and $g$ are polynomials of degree $2$. Suppose that $X$ has two centers $p$, $q\in\Delta$.
 We claim that $p$ and $q$ have opposite orientations. Indeed, let $\ell\subset\mathbb{R}^2$ be the unique straight line such that $p$, $q\in\ell$. Let also $a$, $b$, $c\in\mathbb{R}$ be such that $\ell$ is given by $\ell(x,y)=ax+by+c=0$. Let
\begin{equation}\label{36}
	H(x,y)=\left<X(x,y),(a,b)\right>=a(4x^2-1)f(x,y)+b(4y^2-1)g(x,y).
\end{equation}
Since $p$, $q\in\ell$ are singularities of $X$, it follows that they are also contact points between $X$ and $\ell$. Moreover, since $\deg\ell=1$ and $\deg X=4$, it follows that $X$ can have at most two other contact points. We claim that if such contact point exist, then it is not between $p$ and $q$. Indeed, suppose first that $a=0$, i.e. that $\ell$ is a horizontal line given by $y=y_0$, for some $y_0\in(-{1}/{2},{1}/{2})$. Let $(x,y_0)\in\ell$. It follows from \eqref{36} that $H(x,y_0)=0$ if, and only if,
	\[(4y_0^2-1)g(x,y_0)=0.\]
Since $-{1}/{2}<y_0<{1}/{2}$, it follows that $4y_0^2-1\neq0$ and thus $H(x,y_0)=0$ if, and only if, $g(x,y_0)=0$. Hence, $X$ has at most two contact points with $\ell$. But since $p$, $q\in\ell$ are already contact points of $X$, it follows that $X$ has no other contact points with $\ell$. The case $b=0$ follows similarly. Hence, for now on, we can assume that $\ell$ is neither horizontal, nor vertical. In particular, it follows that $\ell$ must intersect with the straight lines
	\[2x-1=0, \quad 2x+1=0, \quad 2y-1=0, \quad 2y+1=0.\]
Suppose first that this intersection occurs in four different points $q_1,\dots,q_4\in\mathbb{R}^2$. See Figure~\ref{Fig1}.
\begin{figure}[h]
	\begin{center}
		\begin{overpic}[height=8cm]{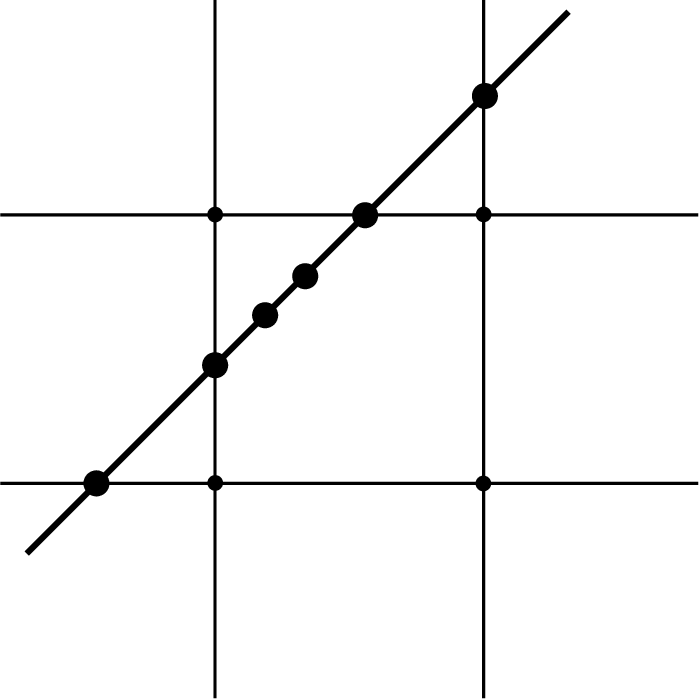} 
			\put(40,52){$p$}
			\put(46,58){$q$}
			\put(102,29){$y=-\frac{1}{2}$}
			\put(102,68){$y=\frac{1}{2}$}
			\put(13,95){$x=-\frac{1}{2}$}
			\put(56,95){$x=\frac{1}{2}$}
			\put(82.5,97){$\ell$}
			\put(72,85){$q_1$}
			\put(33,45){$q_2$}
			\put(53,65){$q_3$}
			\put(13,26){$q_4$}
			\put(32,27){$p_1$}
			\put(70,27){$p_2$}
			\put(71,71){$p_3$}
			\put(32,71){$p_4$}
		\end{overpic}
	\end{center}
	\caption{An illustration of $\ell$.}\label{Fig1}
\end{figure}
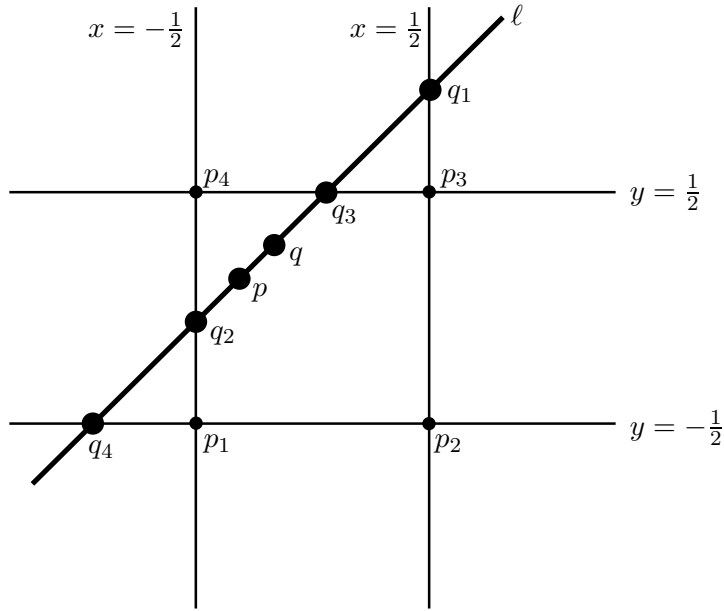
Let $\ell=\infty p\cup pq \cup q\infty$, where $pq$ denotes the segment between $p$ and $q$, $\infty p$ denotes the semi-straight-line containing $p$ and not $q$, and $q\infty$ be the semi-straight-line containing $q$ and not $p$. We claim that $X$ has a contact point in $\infty p$ and in $q\infty$. Indeed, since $p$, $q\in\Delta$ are singularities of $X$, with $X$ given by \eqref{37}, it follows that $f(p)=g(p)=0$. Since $f$ and $g$ have degree two, it follows that $Y=(f,g)$ cannot have other contact points with $\ell$. Therefore, $f$ and $g$ are nonzero and have constant sign on $\infty q\backslash\{q\}$ and on $p\infty\backslash\{p\}$. We will now focus our attention on $q\infty$. The proof for $\infty p$ is similar. Let $w\in q\infty\backslash\{q\}$ be sufficiently close to $q$. Since $q$ is a non-degenerate center of $X$, i.e. $q$ is a center such that $\det DX(q)>0$, we claim that the module of the angle between $X(w)$ and $\ell$ converges to $\frac{1}{2}{\pi}$ as $w\to q$. Indeed, after a translation, linear change of variables and a rescaling of time, we can suppose in this new set of variables $(x,y)\mapsto(u,v)$, that $q$ is the origin, $\ell$ is given either by $u=0$ or $v=\alpha u$, for some $\alpha\in\mathbb{R}$, and that $X=(\mathcal{P},\mathcal{Q})$ is given by,
	\[\mathcal{P}(u,v)=-v+F(u,v), \quad \mathcal{Q}(u,v)=u+G(u,v),\]
where $F$ and $G$ have neither linear nor constant terms (recall \eqref{21}). If $\ell$ is given by $u=0$, then
	\[\lim\limits_{v\to0}\frac{\mathcal{Q}(0,v)}{\mathcal{P}(0,v)}=\lim\limits_{u\to0}\frac{G(0,v)}{-v+F(0,\alpha v)}=0,\]
with the last equality following from the fact that $F$ and $G$ have neither linear nor constant terms. Similarly, if $\ell$ is given by $v=\alpha u$ for some $\alpha\in\mathbb{R}$, then
	\[\lim\limits_{u\to0}\frac{\mathcal{P}(u,\alpha u)}{\mathcal{Q}(u,\alpha u)}=\lim\limits_{u\to0}\frac{-\alpha u+F(u,\alpha u)}{u+G(u,\alpha u)}=-\alpha.\]
In either case, it follows that $X(w)$ is becomes orthogonal to $\ell$ as $w\in q\infty\backslash\{q\}$ approaches $q$. This proves the claim.

Therefore, back to our original coordinates, if $(a,b)$ is a normal vector of $\ell$, it follows that if $w\in q\infty\backslash\{q\}$ is sufficiently close to $q$, then exactly one of the following two statements hold.
\begin{enumerate}[label=(\alph*)]
	\item $\operatorname{sgn}(P(w))=\operatorname{sgn}(a)$ and $\operatorname{sgn}(Q(w))=\operatorname{sgn}(b)$;
	\item $\operatorname{sgn}(P(w))=-\operatorname{sgn}(a)$ and $\operatorname{sgn}(Q(w))=-\operatorname{sgn}(b)$.
\end{enumerate}
See Figure~\ref{Fig2}.
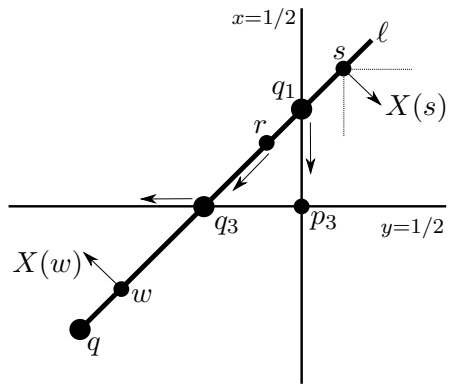
\begin{figure}[h]
	\begin{center}
		\begin{overpic}[height=5cm]{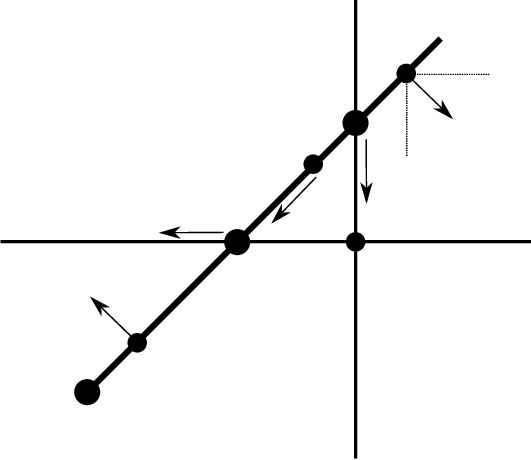} 
			\put(18,8){$q$}
			\put(28,19){$w$}
			\put(47,36){$q_3$}
			\put(60,67){$q_1$}
			\put(56,58){$r$}
			\put(74,75){$s$}
			\put(85,35){${\scriptstyle y={1}/{2}}$}
			\put(51,83){${\scriptstyle x={1}/{2}}$}
			\put(84,79){$\ell$}
			\put(1,26){$X(w)$}
			\put(86,62){$X(s)$}
			\put(69,37){$p_3$}
		\end{overpic}
	\end{center}
	\caption{An illustration of the rotation of $X$ along $\infty q$.}\label{Fig2}
\end{figure}
Since $Y=(f,g)$ cannot have other contact points with $\ell$, we recall that $f$ and $g$ have nonzero constant sign on $q\infty\backslash\{q\}$. Therefore, it follows from \eqref{37} that $P$ and $Q$ change signs if, and only if, $\ell$ intersects the straight lines $x={1}/{2}$ or $y={1}/{2}$, respectively. Since both intersections do exists (given by $q_1$ and $q_3$, respectively), it follows that as we move along $q\infty$, $Q$ changes sign at $q_3$ and then $P$ changes signs at $q_1$. Hence, $X$ has a contact point $r\in q_3q_1\subset q\infty$. See Figure~\ref{Fig2}. Similarly, $X$ has a another contact point $\overline{r}\in q_4q_2\subset\infty p$. These points, together with $p$ and $q$ provides the maximum number of contact points that $X$ can have with $\ell$. Therefore, we conclude that there is no contact point between $p$ and $q$ and thus $p$ and $q$ are centers of opposite orientations. 

If the points $q_1,\dots,q_4$ are not distinct, then $\ell$ intersects at least one of the singularities 
	\[p_1=\left(- {1}/{2},- {1}/{2}\right), \quad p_2=\left( {1}/{2},- {1}/{2}\right), \quad p_3=\left( {1}/{2}, {1}/{2}\right), \quad p_4=\left(- {1}/{2}, {1}/{2}\right).\]
For example $q_1=q_3$ if, and only if, $p_3\in\ell$ (i.e. if and only if $q_1=q_3=p_3$). See Figure~\ref{Fig1}. Since $p_3$ is a singularity, it follows that we again have a contact point $r=p_3\in q\infty$. See Figure~\ref{Fig2}.

Hence, in any case it follows that any two centers of $X$ in $\Delta$ must have opposite orientations and thus there is at most two of them. This finishes the proof. \end{proof}

\begin{proof}[Proof of Theorem~\ref{T2}] 
	Statement~$(a)$ is a straightforward consequence of Proposition~\ref{pr:sep} because  
		if $X=(P,Q)\in\mathcal{X}{^0_2}$ it is given by,
		\begin{equation*}
			P(x,y)=(4x^2-1)(a_{00}+a_{01}y), \quad Q(x,y)=(4y^2-1)(b_{00}+b_{10}x).
		\end{equation*}	
		Statements $(b)$ and $(c)$ follow directly from Propositions~\ref{P7} and \ref{P9}\,$(c)$, respectively.
 \end{proof}

\section*{Acknowledgments}

The first author is supported by  Ministry of Science and Innovation--State Research Agency of the Spanish Government through grants PID2022-136613NB-I00  and it is also supported by the grant 2021-SGR-00113 from AGAUR of Generalitat de Catalunya. The second author is supported by São Paulo Research Foundation (FAPESP) grants number 2022/03801-3, 2020/04717-0. The third author is supported by São Paulo Research Foundation (FAPESP) grants number 2019/10269-3, 2021/01799-9 and 2022/14353-1.

\end{document}